\documentclass[11pt]{amsart}

\usepackage{amssymb}   
\usepackage{amsmath}
\usepackage{amsthm}

\usepackage{times}
\usepackage{color}
\newtheorem{Theorem}{Theorem}[section]

\newtheorem{Proposition}[Theorem]{Proposition}

\newtheorem{Conjecture}[Theorem]{Conjecture}

 \theoremstyle{definition}
\newtheorem{Definition}[Theorem]{Definition}

\newtheorem{Remarks}[Theorem]{Remarks}
\newtheorem{Examples}[Theorem]{Examples}

\numberwithin{equation}{section}

\newcommand{\Pic}{\operatorname{Pic}}
\newcommand{\Bl}{\operatorname{Bl}}

\newcommand{\Sing}{\operatorname{Sing}}

\renewcommand{\O}{{\mathcal O}}

\newcommand{\I}{{\mathcal I}}

\newcommand{\p}{{\mathbb P}}

\newcommand{\codim}{\operatorname{codim}}

\newcommand{\map}{\dasharrow}

\def\leq{\leqslant}
\def\geq{\geqslant}

\def\bibaut#1{{\sc #1}}

\def\phi{\varphi}

\begin{document}

\begin{abstract}Small codimensional embedded manifolds defined by equations of small degree are Fano and covered by lines. They are complete intersections exactly when the variety of lines through a general point is so and has the right codimension. This allows us to prove the Hartshorne Conjecture for manifolds defined by quadratic equations and to obtain the list of such Hartshorne manifolds. 
\end{abstract}
\title[Manifolds covered by lines]{Manifolds covered by lines and the Hartshorne Conjecture for quadratic manifolds}
\author[Paltin Ionescu]{Paltin Ionescu*}
\address{\sc University of Bucharest, Faculty of Mathematics and Computer Science,\linebreak 14 Academiei St., 010014 Bucharest\newline \noindent and \newline\indent Institute of Mathematics of the Romanian Academy, P.O. Box 1-764, 014700 Bucharest\\ Romania}
\email{Paltin.Ionescu@imar.ro}
\author[Francesco Russo]{Francesco Russo}
\address{\sc Dipartimento di Matematica e Informatica\\
Universit\` a degli Studi di Catania\\
Viale A. Doria, 6\\
95125 Catania\\ Italy}
\email{frusso@dmi.unict.it}

\thanks{*Partially  supported  by the Italian Programme ``Incentivazione alla mobilit\`{a} di studiosi stranieri e italiani residenti all'estero"}
\maketitle

\section*{Introduction}

The present paper is a natural sequel to our previous work, see \cite{Ru, IR, IR2} and also \cite{BI}. The geometry of the variety of lines passing through the general point of an embedded projective manifold is investigated, in the framework of Mori Theory. Our main results are a proof of the Hartshorne Conjecture on complete intersections for manifolds defined by quadratic equations and a description of all extremal cases.

We consider $n$-dimensional irreducible non-degenerate complex projective manifolds $X\subset \p^{n+c}$. We call $X$ a {\it prime Fano manifold of index $i(X)$} if its Picard group is generated by the hyperplane section class $H$ and $-K_X=i(X)H$ for some positive integer $i(X)$. One consequence of Mori's work \cite{Mori} is that, for $i(X) > \frac{n+1}{2}$, $X$ is {\it covered by lines}, i.e.\ through each point of $X$ there passes a line, contained in $X$. As noticed classically, Fano {\it complete intersections} with $i(X)\geq 2$ are also covered by lines. The ``biregular part'' of Mori Theory
(no singularities, no flips, ...), see \cite{Mori2, De, Ko}, provides the natural setting for studying manifolds covered by lines. For instance, as first noticed in \cite{BSW}, when the dimension of the variety of lines passing through a general point is at least $\frac{n-1}{2}$, there is a Mori contraction of the covering family of lines. Moreover, its general fiber (which is still covered by lines) has cyclic Picard group, thus being a prime Fano manifold. For prime Fanos, the study of covering families of lines is nothing but the classical aspect in the theory of the {\it variety of minimal rational tangents}, developed by Hwang and Mok in a remarkable series of papers, see e.g. \cite{HM, HM2, HM3, Hwang}. We also recall that lines contained in $X$ play a key role in the proof of an important  result due to Barth--Van de Ven and Hartshorne, cf. \cite{Barth}. It states that $X$ must be a complete intersection when its dimension is greater than a suitable (quadratic) function of its degree.

Prime Fanos of high index other than complete intersections are quite rare. 
Moreover, all known examples of prime Fanos of high index are either complete intersections or {\it quadratic}, i.e. scheme theoretically defined by quadratic equations. Mumford, \cite{Mum}, was the first to call the attention to the fact that many special but highly interesting embedded manifolds are quadratic. One crucial remark we made is that prime Fanos of high index are embedded with {\it small codimension}, which naturally leads us to the famous {\it Hartshorne Conjecture}, see~\cite{Ha}:
if $n\geq 2c+1$, $X$ should be a complete intersection.
This conjecture is already very difficult even in the special case of (prime) Fano manifolds. Early contributions related to the conjecture came by with a topological flavor. First the celebrated Fulton--Hansen Connectivity Theorem \cite{FH}, followed by Zak's Linear Normality Theorem and his equally famous Theorem on Tangencies, see \cite{Zak}. Then, the beautiful result of Faltings \cite{Faltings}, later improved by Netsvetaev \cite{Net}, showing that $X$ is a complete intersection when the {\it number} of equations scheme theoretically defining it is small. Another development, suggested by classical work of Severi, is due to Bertram--Ein--Lazarsfeld \cite{BEL}. They characterize complete intersections in terms of the {\it degrees} of the first $c$ equations defining $X$ (in decreasing order).

Consider now $x\in X$ a general point and let $\mathcal{L}_x$ denote the variety of lines through $x$, contained in $X$. Note that $\mathcal{L}_x$ is naturally embedded in $\p^{n-1}=$ the space of tangent directions at $x$. Let $a:=\dim(\mathcal{L}_x)$ and note that for prime Fanos, $i(X)=a+2$. When $a\geq \frac{n-1}{2}$, $\mathcal{L}_x\subset \p^{n-1}$ is smooth, irreducible and, due to a key result by Hwang \cite{Hwang}, non-degenerate. Our Theorem~\ref{schemeLxrefined}, inspired by \cite[Corollary~4]{BEL}, shows that we can relate the equations of $\mathcal{L}_x\subset \p^{n-1}$ to those of $X\subset \p^N$. In particular, manifolds of small codimension and defined by equations of small degree are covered by lines; moreover, $X\subset \p^N$ is a complete intersection if and only if $\mathcal{L}_x\subset \p^{n-1}$ is so and its codimension is the right one. In case $X$ is quadratic, we combine the above with Faltings' Criterion \cite{Faltings} to deduce that, when $n\geq 2c+1$, $\mathcal{L}_x\subset \p^{n-1}$ is a complete intersection.  Next, we appeal to a different ingredient, leading to a proof of the Hartshorne Conjecture for quadratic manifolds, see Theorem~\ref{HaCo}. The necessary new piece of information is provided by the projective second fundamental form which, due to the Fulton--Hansen Connectivity Theorem, turns out to be of maximal dimension. We would like to point out that, working with local differential geometric methods, in the spirit of \cite{GH, IL}, Landsberg in \cite{Landsberg} proved that when $n\geq 3c+b-1$ a (possibly singular) quadratic variety $X$ is a complete intersection, where $b=\dim({\rm Sing}(X))$.

Using Netsvetaev's Theorem \cite{Net}, we show that the only quadratic manifolds $X\subset \p^\frac{3n}{2}$ which are not complete intersections are $\mathbb G(1,4)\subset\p^9$ and the spinorial manifold $S^{10}\subset\p^{15}$, see our Theorem~\ref{HV}.
\newpage

\section{Preliminaries}
\begin{enumerate}
\item[$(*)$] {\it Setting, terminology and notation}
\end{enumerate}

Throughout the paper we consider $X\subset \p^N$ an irreducible complex projective manifold
of dimension $n\geq 1$. $X$ is assumed to be non-degenerate and $c$ denotes its codimension, so that $N=n+c$. We also suppose that $X$ is scheme theoretically an intersection of $m$ hypersurfaces of degrees $d_1\geq d_2\geq\cdots\geq d_m$. It is implicitly assumed that $m$ is minimal, i.e. none of the hypersurfaces contains the intersection of the others. We put
$d:=\sum_{i=1}^{c}(d_i-1)$. $X\subset\p^N$ is called {\it quadratic} if it is scheme theoretically an intersection of quadrics (i.e. $d_1=2$). Note that this happens precisely when $d=c$.

For $x\in X$ we let $\mathbf{T}_xX$ denote the (affine) Zariski tangent space to $X$ at $x$, and write $T_xX$ for its projective closure in $\p^N$.
$H$ denotes a hyperplane section (class) of $X$. As usual, $K_X$ stands for the canonical class of $X$. Also, if $Y\subset X$ is a submanifold, we denote by $N_{Y/X}$ its normal bundle. For a vector bundle $E$, $\p(E)$ stands for its projectivized bundle, using Grothendieck's convention.

We let $SX\subset \p^N$ be the {\it secant variety} of $X$, that is the closure of the locus of secant lines. The {\it secant defect} of $X$ is the (nonnegative) number $\delta=\delta(X):=2n+1-\dim(SX)$. We say $X$ is {\it secant defective} when $\delta > 0$.


$X\subset \p^N$ is {\it conic-connected} if any two general points $x, x'\in X$ are contained in some conic $C_{x,x'}\subset X$. Conic-connected manifolds are secant defective. Classification results for conic-connected manifolds are obtained in \cite{IR}, working in the general setting of rationally connected manifolds.

For a general point $x\in X$, we denote by $\mathcal L_x$ the (possibly empty) scheme of lines contained in $X$ and passing through $x$. We say that $X\subset \p^N$ is {\it covered by lines} if $\mathcal L_x$ is not empty for $x\in X$ a general point. We refer the reader to \cite{De, Ko} for standard useful facts about the deformation theory of rational curves; we shall use them implicitly in the simplest case, that is lines on $X$.

Recall that $X$ is Fano if $-K_X$ is ample. The {\it index} of $X$, denoted by $i(X)$, is the largest integer $j$ such that $-K_X = jA$ for some ample divisor $A$.

\begin{enumerate}
\item[$(**)$] {\it The projective second fundamental form}
\end{enumerate}

We recall some general results which 
are probably well known to the experts but for which we are unable to provide a proper reference.

 There are several possible equivalent
definitions of the projective second fundamental form
$|II_{x,X}|\subseteq\p(S^2(\mathbf{T}_xX))$ of an irreducible
projective variety $X\subset\p^N$ at a general point $x\in X$, see
for example \cite[3.2 and end of Section~3.5]{IL}. We use the
one related to tangential projections, as in \cite[Remark
3.2.11]{IL}.

 Suppose $X\subset\p^N$ is non-degenerate, as always,
 let $x\in X$ be a general point and
consider the projection from $T_xX$ onto a disjoint $\p^{c-1}$
\begin{equation}\label{tangentdef}
\pi_x:X\map W_x\subseteq\p^{c-1}.
\end{equation}
The map $\pi_x$ is associated to the linear system of hyperplane
sections cut out by hyperplanes containing $T_xX$, or equivalently
by the hyperplane sections singular~at~$x$.

Let $\phi:\Bl_xX\to X$ be the blow-up of $X$ at $x$, let
\[E=\p((\mathbf{T}_xX)^*)=\p^{n-1}\subset\Bl_xX\] be the exceptional
divisor and let $H$ be a hyperplane section of $X\subset\p^N$. The
induced rational map $\widetilde{\pi}_x:\Bl_xX\map\p^{c-1}$ is
defined as a rational map along $E$ since $X\subset\p^N$ is not a linear space.
The restriction of
$\widetilde{\pi}_x$ to $E$ is given by a linear system in
$|\phi^*(H)-2E|_{|E}\subseteq|-2E_{|E}|=|\O_{\p((\mathbf{T}_xX)^*)}(2)|=\p(S^2(\mathbf{T}_xX))$.

\begin{Definition}
 The {\it second fundamental form}
$|II_{x,X}|\subseteq\p(S^2(\mathbf{T}_xX))$ of an irreducible
non-degenerate variety $X\subset\p^N$ of dimension $n\geq 2$ at a
general point $x\in X$ is the non-empty linear system of quadric
hypersurfaces in $\p((\mathbf{T}_xX)^*)$ defining the restriction of
$\widetilde{\pi}_x$ to $E$.
\end{Definition}
 
Clearly  $\dim(|II_{x,X}|)\leq c-1$ and
$\widetilde{\pi}_x(E)\subseteq W_x\subseteq\p^{c-1}$. From this
point of view the base locus on $E$ of the second fundamental form
$|II_{x,X}|$ consists of {\it asymptotic directions}, i.e.\ of
directions associated to lines having a contact of order at least three with
$X$ at $x$. For example, as we shall see, when $X\subset\p^N$ is quadratic,  the base locus of the second fundamental form
consists of points representing tangent lines contained in $X$ and passing
through $x$, so that it is exactly (even scheme theoretically) $\mathcal L_x$.

\begin{Proposition}\label{secondform} Let $X\subset\p^N$ be
a smooth irreducible non-degenerate variety of secant defect
$\delta\geq 1$.
Then $\dim(|II_{x,X}|)=c-1$ for $x\in X$ a general point.
\end{Proposition}
\begin{proof}
Let notation be as above. It is sufficient to show
that $\dim(\widetilde{\pi}_x(E))=n-\delta$ because
$\widetilde{\pi}_x(E)\subseteq W_x$ and 
$W_x\subseteq\p^{c-1}$ is a non-degenerate variety, whose  dimension
is  $n-\delta$ by the Terracini Lemma.

Let $TX=\bigcup_{x\in X}T_xX$ be the tangential variety of $X$. The
following formula holds
\begin{equation}\label{dimension}
\dim(TX)=n+1+\dim(\widetilde{\pi}_x(E)),
\end{equation}
 see \cite{Terracini} (or
\cite[5.6, 5.7]{GH}  and \cite[Proposition 3.13.3]{IL} for a modern reference).

The variety $X\subset\p^N$ is smooth and secant defective, so that
$TX=SX$ by a theorem of Fulton and Hansen, \cite{FH}. Therefore
$\dim(TX)=2n+1-\delta$ and from \eqref{dimension} we get
$\dim(\widetilde{\pi}_x(E))=n-\delta$, as claimed.
\end{proof}

\section{Manifolds covered by lines}

Let $X\subset \p^N$ be as in ($*$). The following examples of manifolds covered by lines are of relevance to us.

\begin{Examples}\label{ex}

\begin{enumerate}
\item $X\subset \p^N$ a Fano complete intersection with  $i(X)\geq 2$. 

\item (Mori) $X\subset \mathbb{P}^N$ a Fano manifold with ${\rm Pic}(X)=\mathbb{Z}\langle H\rangle$ and  $i(X)> \frac {n+1}2$.

\item $X\subset \p^N$ a conic-connected manifold, different from the Veronese variety $v_2(\p^n)$ or one of its isomorphic projections.

\end{enumerate}
\end{Examples}

Fix some irreducible component, say $\mathcal F$, of the Hilbert scheme of lines on $X\subset \p^N$, such that $X$ is covered by the lines in $\mathcal F$. Put $a=: \deg (N_{\ell/X})$ where $[\ell] \in \mathcal F$. Note that
$a\geq 0$ and $a= \dim (\mathcal{F} _x)$, where $x\in X$ is a general point and     $\mathcal{F}_x=\{[\ell] \in\mathcal F\mid x\in \ell\}$. Moreover, we may view $\mathcal{F}_x$ as a closed subscheme of $\p((\mathbf T_xX)^*)\cong \p^{n-1}$.

When the dimension of $\mathcal{F}_x$ is large, the study of manifolds covered by lines is greatly simplified by the following two facts:

First, we may reduce, via a Mori contraction, to the case where the Picard group is cyclic; this is due to Beltrametti--Sommese--Wi\'sniewski, see \cite{BSW}.
Secondly, the variety $\mathcal{F}_x\subset \p^{n-1}$
inherits many of the good properties of $X\subset \p^N$; this is due to Hwang, see \cite{Hwang}. See \cite{BI} for an application of these principles.
\begin{Theorem}\label{reduction}
Assume $a\geq \frac {n-1}2$. Then the following results hold:
\begin{enumerate}
\item {\rm (\cite{BSW})} There is a Mori contraction, say ${\rm cont}_{\mathcal F} : X\to W$, of the lines from $\mathcal F$; let $F$ denote a general fiber of ${\rm cont}_{\mathcal F}$ and let $f$ be its dimension;

\item {\rm (\cite{Wisniewski})} ${\rm Pic}(F) = \mathbb{Z} \langle H_F\rangle$, $i(F)= a+2$ and $F$ is covered by the lines from $\mathcal F$ contained in $F$;

\item {\rm (\cite{Hwang})} $\mathcal {F}_x\subseteq \mathbb{P}^{f-1}$ is smooth irreducible non-degenerate. In particular, $F$ has only one maximal irreducible covering family of lines.
\end{enumerate}
\end{Theorem}

The following very useful result was proved in \cite[Corollary 4]{BEL}. It will play a crucial role in what follows.

\begin{Theorem}[\cite{BEL}]\label{criterion} Let $X\subset\p^N$ be as in {\rm($*$)}. If 
$$-K_X=\O_X(n+1-d),$$
then $X\subset\p^N$ is a complete intersection of type $(d_1,\ldots,d_c)$.
\end{Theorem}
\begin{proof} Under our hypothesis 
\begin{align*}h^{n+1}(\I_X(d-n-1))&=h^n(\O_X(d-n-1))\\
&=h^0(\omega_X\otimes\O_X(n+1-d))=h^0(\O_X)\neq 0.\end{align*}
Thus $X\subset\p^N$ fails to be $d$-regular and we can apply \cite[Corollary 4]{BEL}.
\end{proof}

Now we can prove the following theorem, showing that $\mathcal{L}_x\subset \p^{n-1}$ ``behaves better'' than
$X\subset \p^N$. Indeed we can control the number of equations defining $\mathcal{L}_x$, as in the classical Example~\ref{ex}~(1), but without assuming $X$ to be a complete intersection.

\begin{Theorem}\label{schemeLxrefined} Let $X\subset\p^N$ be as in {\rm($*$)}. For $x\in X$ a general point, put $a=\dim(\mathcal L_x)$. Then the following results hold: 

\begin{enumerate}

\item If $\mathcal L_x\subset\p^{n-1}$ is nonempty, it is set theoretically defined by (at most) $d$ equations; in particular, we have $a\geq n-1-d$.

\item If $X$ is quadratic and $\mathcal L_x\subset\p^{n-1}$ is nonempty, $\mathcal L_x$ is scheme theoretically defined by (at most) $c$ quadratic equations.

\item If $d\leq n-1$, then $\mathcal L_x\neq\emptyset$; assume moreover that $n\geq c+2$ if $X$ is quadratic. Then

\begin{enumerate}

\item $X\subset\p^N$ is a Fano manifold with $\Pic(X)\simeq\mathbb Z\langle H\rangle$ and $i(X)=a+2$;

\item the following conditions are equivalent:

\begin{enumerate}

\item[{\rm (i)}]  $X\subset\p^N$ is a complete intersection;
\item[{\rm (ii)}] $\mathcal L_x\subset\p^{n-1}$ is a complete intersection of codimension $d$;
\item[{\rm (iii)}] $a=n-1-d$.

\end{enumerate}
\end{enumerate}

\item Assume that $\Pic(X)\simeq\mathbb Z\langle H\rangle$, $a\geq \frac{n-1}{2}$ and $\mathcal L_x\subset\p^{n-1}$ is a complete intersection. Then $X$ is conic-connected, $a\leq n-c-1$ and $n\geq 2c+1$.


\end{enumerate} 
\end{Theorem}
\begin{proof} 
Let us take a closer look at the proof of the previous result given in \cite{BEL}. Since $X\subset\p^N$ is scheme theoretically defined by equations of degree $d_1\geq d_2\geq\cdots\geq d_m$, we can choose $f_i\in H^0(\p^N,\I_X(d_i))$, $i=1,\ldots,c$,
such that, letting $Q_i=V(f_i)\subset\p^N$, we obtain the complete intersection scheme
$$Y=Q_1\cap\cdots\cap Q_c=X\cup X',$$
where $X'$ (if nonempty) meets $X$ in a divisor.

Thus for $x\in X$ general we have that  a line $l$ passing through $x$ is contained in $X\subseteq Y$ if and only if $l$ is contained in $Y$, that is $\mathcal L_x(X)\subset\p^{n-1}$ coincides set theoretically with $\mathcal L_x(Y)\subset\p^{n-1}$.
Suppose, as recalled above, that
$$Y=V(f_1,\ldots, f_c)=X\cup X'\subset\p^N$$ with $f_i$ homogeneous polynomials of degree $d_i$ and let $x\in X$ be a general point. Without loss of generality we can suppose $x=(1:0:\ldots:0)$. Then in the affine space $\mathbb A^N$ defined by $x_0\neq 0$, we have
$f_i=f_i^1+f_i^2+\cdots+f_i^{d_i}$, with $f_i^j$ homogeneous of degree $j$ in the variables $(y_1,\ldots, y_N)$, where  $y_l=\frac{x_l}{x_0}$ for every $l\geq 1$ and $x=(0,\ldots,0)$. The equations of $\p^{n-1}$ inside
 $\p^{N-1}=\p((\mathbf T_x\p^N)^*)$ are exactly $f_1^1=\cdots=f_c^1=0$ while the
equations of $\mathcal L_x(Y)$ inside $\p^{n-1}$ are  $f_1^2=\cdots=f_1^{d_1}=\cdots=f_c^2=\cdots=f_c^{d_c}=0$, which are exactly $\sum_{i=1}^c(d_i-1)=d$. In particular $\mathcal L_x\subset\p^{n-1}$,
if not empty, is set theoretically
defined by these equations, proving (1).

To prove (2), assume that $X$ is quadratic. Keeping the notation above, we see that $\mathcal L_x\subset\p^{n-1}$ is scheme theoretically defined by the equations
$f_1^2,\ldots, f_m^2$ (modulo the ideal generated by $f_1^1,\ldots, f_c^1$). But the same homogeneous quadratic equations define the affine scheme $X\cap \mathbf{T}_xX\subset \mathbf{T}_xX$. In particular, the pointed affine cone over $\mathcal L_x\subset\p^{n-1}$ and the scheme $(X\cap \mathbf{T}_xX) \setminus \{x\}$ coincide. Consider now, as above, $Y=V(f_1,\ldots, f_c)=X\cup X'$. We have that the pointed affine cone over $\mathcal L_x(Y)\subset\p^{n-1}$ and the scheme $(Y\cap \mathbf{T}_xY) \setminus \{x\}$ coincide and are scheme theoretically defined by (at most) $c$ quadratic equations. But $X$ and $Y$ coincide in a neighborhood of $x$, hence the pointed affine cones over $\mathcal L_x(X)$ and $\mathcal L_x(Y)$ also coincide. As $\mathcal L_x(X)$ is smooth, this implies that $\mathcal L_x(Y)$ is also smooth and so they coincide as schemes (see also \cite{Rulines} for more details and for other proofs of this key result). This shows (2).

Under the hypothesis of (3), we deduce 
$a\geq n-1-d\geq 0.$

The assertion about the Picard group in (a) follows from the Barth--Larsen Theorem, see \cite{BL}.
Since $x\in X$ is general and since $\mathcal L_x\neq\emptyset$ by the above argument, for every line $l$ passing through $x$ we have  $$-K_X\cdot l= 2+a\geq 2,$$
concluding the proof of (a).

To show (b), let us remark that by the previous discussion (i) implies (ii); (ii) implies (iii) is obvious.  Theorem \ref{criterion} shows that (iii) implies (i).

We now pass to part (4). Since $a\geq \frac{n-1}{2}$,
$\mathcal{L}_x\subset \p^{n-1}$ is smooth irreducible and non-degenerate by Theorem \ref{reduction} (3).
As $\mathcal{L}_x\subset \p^{n-1}$ is a non-degenerate complete intersection of dimension greater or equal to one less its codimension, it follows from the Terracini Lemma that $S\mathcal{L}_x=\p^{n-1}$, see e.g. \cite{Russo}. Now we may apply \cite[Theorem~3.14]{HK} to infer that $X$ is conic-connected. In particular, we have $\delta(X)>0$. From Proposition~\ref{secondform} it follows that $\dim|II_{x,X}|=c-1$. As $\mathcal{L}_x$ is contained in the base-locus of $|II_{x,X}|$, it follows that $\mathcal{L}_x$ is contained in at least $c$ linearly independent quadrics. But, being a complete intersection, the number of such quadrics can not exceed the codimension of $\mathcal{L}_x$ in $\p^{n-1}$.
This means that $c\leq n-1-a$, which is what we wanted.\end{proof}

\section{Around the Hartshorne Conjecture; the quadratic case}

In \cite{Ha}, Hartshorne  made his now famous conjecture:
\begin{Conjecture}[HC] If $X\subset \p^N$ is as in {\rm($*$)} and  $n\geq2c+1$ then $X$ is a complete intersection.
\end{Conjecture}
This is still widely open, even for $c=2$.
We would like to state explicitly several weaker versions.
\begin{Conjecture}[HCF] Assume that $n\geq 2c+1$ and $X$ is Fano; then $X$ is a complete intersection.
\end{Conjecture}

\begin{Conjecture}[HCL] Assume that $X\subset\p^N$ is covered by lines with  $\dim(\mathcal L_x)\geq\frac{n-1}{2}$ and let $T$ be the span of $\mathcal L_x$ in $\p^{n-1}$. If $\dim(\mathcal L_x)>2\codim(\mathcal L_x, T)$, then
$\mathcal L_x\subset\p^{n-1}$ is a complete intersection. 
\end{Conjecture}

\begin{Remarks}
\begin{enumerate}
\item When $n\geq 2c+1$, by the Barth--Larsen Theorem ${\rm Pic}(X) \cong \mathbb{Z}\langle H\rangle$. In particular $K_X=bH$ for some integer $b$. So, saying that $X$ is Fano means exactly that $b < 0$; this happens, for instance, if $X$ is covered by lines.
\item The (HCF) holds when $c=2$, see \cite{BC}.
\item By Theorem \ref{reduction}, the  (HCL) concerns Fano manifolds $X\subset\p^N$ with $\Pic(X)\simeq\mathbb Z\langle H\rangle$ and of index $i(X)\geq\frac{n+3}{2}$. 
\item Manifolds of (very) small degree are known to be complete intersections, cf. \cite{Barth}. The (HCF) would yield the following optimal result, see \cite{Io}.

\end{enumerate}
\end{Remarks}
\begin{Conjecture}
If $n\geq degree(X)+1$, then $X$ is a complete intersection, unless it is projectively equivalent to $\mathbb G(1,4)\subset\p^9$.
\end{Conjecture}
Note that the Segre embedding $\p^1\times \p^{n-1}\subset \p^{2n-1}$ has degree $n$ and is not a complete intersection if $n\geq 3$.
 
 Prime Fano manifolds of high index tend to be complete intersections. In the next proposition we propose such bounds on the index and show how they would follow from the (HC). However, we expect this kind of result to be easier to prove than the general (HC). 
\begin{Proposition}\label{index} 
\begin{enumerate}
\item

Let $X\subset\p^N$ be a Fano manifold with $\Pic(X)\simeq\mathbb Z\langle H\rangle$ and of index $i(X)\geq\frac{2n+5}{3}$. If the {\rm(HCL)} and the {\rm(HCF)} are true,
then $X$ is a complete intersection.
\item The same conclusion holds assuming only the {\rm(HCF)}, but asking instead that $i(X)\geq \frac{3(n+1)}{4}$.
\end{enumerate}
\end{Proposition}
\begin{proof} (1) Since $i(X)>\frac{n+1}{2}$ we get that $X$ is covered by lines by Example~\ref{ex}~(2). We have $\dim(\mathcal{L}_x)=i(X)-2\geq \frac{2n-1}{3}\geq \frac{n-1}{2}$, hence $\mathcal{L}_x\subset \p^{n-1}$ is smooth irreducible non-degenerate by Theorem~\ref{reduction}~(3). Moreover, we have $\dim(\mathcal{L}_x)\geq 2\codim(\mathcal{L}_x, \p^{n-1})+1$. 
Thus $\mathcal{L}_x\subset \p^{n-1}$ is a complete intersection by the (HCL). Next we may apply Theorem~\ref{schemeLxrefined}~(4) to infer that $n\geq 2c+1$ and then use the (HCF).

(2) Since $i(X)> \frac{2n}{3}$ we have that $X$ is conic-connected by \cite{HK}. From \cite[Proposition~3.2]{IR} we infer 
$$\frac{3(n+1)}{4}\leq i(X)\leq\frac{n+\delta}{2}.$$
This yields  $\delta>\frac{n}{2}$ so that $SX=\p^N$ by Zak's Linear Normality Theorem. Therefore $\delta=n-c+1$
and
$$\frac{3(n+1)}{4}\leq i(X)\leq\frac{n+(n-c+1)}{2}$$ 
implies
$n\geq 2c+1$, so the (HCF) applies.
\end{proof}

Note that for complete intersections $X\subset \p^N$, $\mathcal{L}_x\subset \p^{n-1}$ is also a complete intersection. We believe that for manifolds covered by lines the converse should also hold.

\begin{Conjecture}
If $X\subset \p^N$ is covered by lines and $\mathcal{L}_x\subset \p^{n-1}$ is a (say smooth irreducible non-degenerate) complete intersection, then $X$ is a complete intersection too.
\end{Conjecture}
Theorem~\ref{schemeLxrefined}~(4) shows that the above conjecture follows from the (HCF), at least when $\dim(\mathcal{L}_x)\geq \frac{n-1}{2}$.

\medskip

Mumford in his seminal series of lectures \cite{Mum} called the attention to the fact that many interesting embedded manifolds are scheme theoretically defined by quadratic equations.

As a special case of the results in \cite{BEL}, if $X$ is quadratic, we have: 
\begin{enumerate}
\item If $n\geq c-1$, then $X$ is projectively normal;
\item If $n\geq c$, then $X$ is projectively Cohen--Macaulay.
\end{enumerate}
Our main results are:
\begin{Theorem}\label{HaCo} Assume that $X\subset \p^N$ is as in {\rm($*$)} and $X$ is quadratic.
\begin{enumerate}
\item If $n\geq c$, then $X$ is Fano.
\item If $n\geq c+1$, then $X$ is covered by lines. Moreover, $\mathcal{L}_x\subset \p^{n-1}$
is scheme theoretically defined by $c$ independent quadratic equations.
\smallskip
\item If $n\geq c+2$, then the following conditions are equivalent:
\begin{enumerate}
\smallskip
\item[{\rm(i)}] $X\subset \p^N$ is a complete intersection;
\item[{\rm(ii)}] $\mathcal{L}_x\subset \p^{n-1}$ is a complete intersection;
\item[{\rm(iii)}] $\dim(\mathcal{L}_x)=n-1-c$;
\end{enumerate}

\smallskip
\item{\rm(HC)} If $n\geq2c+1$, then $X$ is a complete intersection.
\smallskip
\item If $\Pic(X)\simeq\mathbb Z\langle H\rangle$ and $X$ is Fano of index $i(X)\geq\frac{2n+5}{3}$, then $X$ is a complete intersection.
\end{enumerate}

\end{Theorem}
\begin{proof}
We use the notation in ($*$). Moreover, denote by $a:=\dim(\mathcal{L}_x$). Since $X$ is quadratic, we have $d=c$. 

 Since $X$ is quadratic, $N_{X/\p^N}^*(2)$ is spanned by global sections. Therefore, $\det(N_{X/\p^N}^*)(2c)$ is also spanned and $\det(N_{X/\p^N}^*)(2c+1)$ is ample. From $n\geq c$ it follows that $N+1\geq 2c+1$. We get the ampleness of $-K_X=\det(N_{X/\p^N}^*)(N+1)$, thus proving (1).
 
 From Theorem~\ref{schemeLxrefined}~(3) and (2) we deduce the first part in (2) and that $\mathcal{L}_x\subset \p^{n-1}$ is defined scheme theoretically by at most $c$ quadratic equations. Since $n\geq c+1$, we must have $\delta\geq n-c+1 >0$. From Proposition~\ref{secondform} we infer that the quadratic equations defining  $\mathcal{L}_x$ are independent, proving (2). This also shows that, when $\mathcal{L}_x$ is a complete intersection, it has codimension $c$ in $\p^{n-1}$. From Theorem~\ref{schemeLxrefined}~(3) the equivalence between conditions (i), (ii) and (iii) of part (3) is now clear. 


To prove part (4), assume that $n\geq 2c+1$. We first observe that by Theorem~\ref{schemeLxrefined}~(1) we have $a\geq n-1-c\geq \frac{n-1}{2}$. Next, since $n\geq 2c+1$, $\Pic(X)\simeq\mathbb Z\langle H\rangle$ by \cite{BL}. From Theorem \ref{reduction}(3) $\mathcal{L}_x\subset \p^{n-1}$ is smooth irreducible and non-degenerate. Being defined scheme theoretically by $c\leq \frac{n-1}{2}$ equations, it is a complete intersection by the result of Faltings \cite{Faltings}. By the previous point $X$ is a complete intersection, too. 

Finally, part (5) follows from part (4) and Proposition~\ref{index}.
\end{proof}

\begin{Theorem}\label{HV} Assume that $X\subset \p^N$ is as in {\rm($*$)} and $X$ is quadratic. If $n=2c$ and $X$ is not a complete intersection, then it is projectively equivalent to one of the following:
\begin{enumerate}
\item [(a)] $\mathbb G(1,4)\subset\p^9$, or
\item [(b)] $S^{10}\subset\p^{15}$.
\end{enumerate}
\end{Theorem}
\begin{proof}
We need a refinement of Faltings' Theorem from \cite{Faltings}, due to Netsvetaev, see \cite {Net}.
Since we assumed that $X$ is not a complete intersection, we have $a\geq n-c=\frac{n}{2}\geq \frac{n-1}{2}$, so $\mathcal{L}_x\subset \p^{n-1}$ is smooth irreducible non-degenerate by Theorem~\ref{reduction}~(3). Moreover, by the previous theorem, $\mathcal{L}_x$ is not a complete intersection in $\p^{n-1}$, but it is scheme theoretically defined by $c$ independent quadratic equations. Netsvetaev's Theorem from \cite{Net} applied to $\mathcal{L}_x\subset \p^{n-1}$ gives the inequalities
$$\frac{3n-6}{4}\leq a <\frac{3n-5}{4} .$$
We may assume $n\geq 6$ (otherwise $\mathcal{L}_x$ is a complete intersection), so we deduce that $n=4r+2$ and $a=3r$, for a suitable $r$. If $r > 2$ (equivalently $n > 10$), we would have $a\geq 2(n-1-a)+1$ and $\mathcal{L}_x$ would be a complete intersection by Theorem~\ref{HaCo}~(4). In conclusion $n=6$ or $n=10$. In the first case $i(X)=a+2=5$ and we get case (a) by the classification of del Pezzo manifolds, see \cite{Fujita}. In the second case, $i(X)=8$, leading to case (b) by \cite{Mukai}.
\end{proof}

\begin{Remarks}
\begin{enumerate}
\item If $X\subset \p^{\frac{3n}{2}}$ is not a complete intersection, it is called a Hartshorne manifold. $\mathbb G(1,4)\subset\p^9$ and $S^{10}\subset\p^{15}$ are such examples, the first one being due to Hartshorne himself. So, for quadratic manifolds we proved not only that the (HC) holds, but also that the above examples are the only Hartshorne manifolds. This shows, once again, how mysterious the general case of the (HC) remains!
\item Working with local differential geometric techniques, as in \cite{GH, IL}, Landsberg \cite{Landsberg} proved that a (possibly singular) quadratic variety satisfying $n\geq 3c+b-1$ is a complete intersection, where $b=\dim(\Sing(X))$.

\end{enumerate}

\end{Remarks}

\end{document}